\documentclass[preprint,12pt]{elsarticle}

\usepackage{amsfonts,amssymb,amsmath,amsgen,amsopn,amsbsy,graphicx,epsfig,amsthm}
\usepackage{amscd,bezier,latexsym,mathrsfs,enumerate,multirow}
\usepackage{physics}

\newtheorem{theorem}{Theorem}[section]

\newtheorem{definition}[theorem]{Definition}
\newtheorem{example}[theorem]{Example}

\newtheorem{lemma}[theorem]{Lemma}

\newcommand{\intP}{\text{int} (\mathcal{P})}
\renewcommand{\natural}{\mathbb N}

\journal{Bulletin des sciences math\'ematiques}

\begin{document}

\begin{frontmatter}

\title{Best Approximations on Quasi-Cone Metric Spaces}

\author{Ahmad Hisbu Zakiyudin, Kistosil Fahim*, Nur Millatul Af-Idah, Felix Lyanto Setiawan, Ririana Annisatul Lathifah, Nur Izzah Nurdin,
 Tantri Tarisma} 

\affiliation{organization={Department of Mathematics, Faculty of Science and Analytical Data, Institut Teknologi Sepuluh Nopember},
            addressline={Sukolilo}, 
            city={Surabaya},
            postcode={60111}, 
            state={East Java},
            country={Indonesia}}

\begin{abstract}
This paper focuses on the best approximation in quasi-cone metric spaces, a combination of quasi-metrics and cone metrics, which generalizes the notion of distance by allowing it to take values in an ordered Banach space. We explore the fundamental properties of best approximations in this setting, such as the best approximation sets and the Chebyshev sets.
\end{abstract}


\end{frontmatter}

\section{Introduction and Preliminaries}
\label{sec1:intro}

Approximation theory is one of the important topics in functional analysis, offering powerful
 tools to address problems in optimization, numerical analysis, and applied mathematics (see \cite{Singer1970-fu,goncharov2000,Yuan2022,Dunham1991,Hromadka1987,steffens2006,Iske2018}). This field seeks to understand the concept of the best approximation-an element that minimizes the distance to a given target within a specific set.

Historically, approximation theory was initially developed to find the approximate values of a real-valued function in Euclidean spaces (see \cite{goncharov2000,Iske2018}). Over time, it expanded into more sophisticated settings, such as inner product spaces and normed spaces (see \cite{Singer1970-fu,Iske2018}). The best approximation theory is further discussed in the classical metric spaces setting (see \cite{Khalil1988,narang1983,Digar2024}). Furthermore, researchers have extended the concept to non-symmetric metric spaces, commonly referred to as quasi-metric spaces, which offer a more generalized framework for approximation problems (see \cite{Romaguera2000,Abkar2016}).

Cone metric spaces, introduced as a generalization of classical metric spaces, provide an even broader perspective. In this setting, the distance between points is not always a real number but an element of an ordered Banach space called a cone (see \cite{huang2007}). This extension
 allows for the analysis of more complex structures, with more potential applications. In this cone metric space, the best approximation was also introduced, which was first discussed by Rezapour (see \cite{Rezapour2007}). 

More recently, the combination of quasi-metrics and cone metrics has produced a quasi-cone metric space which offers flexibility in measuring the distance (see \cite{Abdeljawad2009,Yaying2016,alyaari2022}). However, the theory of the best approximation in this space remains unexplored. In this paper, we investigate the concept of best approximations in quasi-cone metric spaces. Specifically, we aim to obtain fundamental results regarding the best approximation sets and the Chebyshev sets.

Now, we present the concept of a cone and some results in \cite{huang2007,Rezapour2007,Yaying2016,alyaari2022}.

\begin{definition}[See Definition 1.1 of \cite{Yaying2016}]
\label{def1}
Let $\mathcal{P}$ be a non-empty subset of a real Banach space $\mathcal{B}$. Then, $\mathcal{P}$ is said to be a cone over $\mathcal{B}$ if and only if
\begin{enumerate}
    \item[(C1)] $\mathcal{P}$ is closed and $\mathcal{P}\neq \{0_\mathcal{B}\}$;
    \item[(C2)] $ax+by\in \mathcal{P}$ for all $x,y\in \mathcal{P}$ and $a,b\in \mathbb{R}^{+}\cup \{0\}$;
    \item[(C3)] $x\in \mathcal{P}$ and $-x\in \mathcal{P}$ implies $x=0_{\mathcal{B}}$.
\end{enumerate}
\end{definition}
A cone can be referred to as the nonnegative space of $\mathcal{B}$. The following example illustrates this concept.
\begin{example} \label{example1.1}
Let $\mathcal{B}=\mathbb{R}$. It is straightforward to show that $\mathcal
{P}=\mathbb{R}^{+}\cup \{0\}$ satisfies axioms (C1), (C2) and (C3) in Definition \ref{def1}. Moreover, this cone can be extended for an Euclidean space $\mathcal{B}=\mathbb{R}^{n}$ with the corresponding cone $\mathcal{P}=\{(x_1,x_2,\dots,x_n)^T\in \mathcal{B}:x_i\geq 0,i=1,2,\dots,n\}.$

\end{example}
Let us now introduce a partial ordering \(\preceq\). Let $\mathcal{P}$ be a cone over a real Banach space $\mathcal{B}$. A partial ordering \(\preceq\) on $\mathcal{B}$ is defined such that for any \(r, s \in \mathcal{B}\), \(s \preceq r\) if and only if \(r - s \in \mathcal{P}\). In addition, \(x \prec y\) denotes \(x \preceq y\) and \(x \neq y\).

Moreover, we denote by $ s \ll r$ if $r-s \in$ int($\mathcal{P}$)\footnote{int($\mathcal{P}$) is an interior of cone $\mathcal{P}$ (see page 85 in \cite{Rezapour2007}).}.

\begin{example}
\label{ex2}
    Let $\mathcal{P}=\{(x_1,x_2,x_3)^T\in \mathcal{B}:x_1,x_2,x_3\geq 0\}$. We can observe that $$ \left(
\begin{array}{c}
1\\
4\\
3
\end{array}
\right) \prec 
\left(
\begin{array}{c}
1\\
4\\
5
\end{array}
\right)\text{ but }\left(
\begin{array}{c}
1\\
4\\
3
\end{array}
\right) \not\ll
\left(
\begin{array}{c}
1\\
4\\
5
\end{array}
\right).$$ 
In particular, we have $(1~4~5)^T-(1~4~3)^T=(0~0~2)^T\in \mathcal{P}$ and thus $(1~4~3)^T\prec(1~4~5)^T$. Moreover,  we also have $(1~4~5)^T-(1~4~3)^T=(0~0~2)^T\notin \intP$ and it implies 
$(1~4~3)^T\not\ll(1~4~5)^T$.
\end{example}
In the following, we always consider that the set $\mathcal{B}$ be a real Banach space and $\mathcal{P}$ be a cone over $\mathcal{B}$. Next, we present the definition of the quasi-cone metric spaces as follows.
\begin{definition}[See Definition 1.3 of \cite{Yaying2016}]
\label{def2}
A function $d:\mathcal{Q} \times \mathcal{Q} \to \mathcal{B}$ is said to be a quasi-cone metric on $\mathcal{Q}$ if for any $r,s,t\in \mathcal{Q}$ the following axioms are satisfied:
\begin{enumerate}
    \item[(QCM1)] $d(r,s)\succeq 0_\mathcal{B}$;
    \item[(QCM2)] $d(r,s)=0_\mathcal{B}$ if and only if $r=s$;
    \item[(QCM3)] $d(r,t) \preceq d(r,s)+d(s,t)$.
\end{enumerate}
A pair $(\mathcal{Q},d)$ is so-called a quasi-cone metric space.
\end{definition}
To illustrate the above equation, we provide an example of a quasi-cone metric space.
\begin{example}[See Example 2.5 of \cite{alyaari2022}]
    \label{ex3}
    Suppose $\mathcal{Q}=\mathbb{R}, \mathcal{B} = \mathbb{R}^2$, and $\mathcal{P}=\{ (a,b) \in \mathcal{B}:a,b \ge 0\}$. Let us define $d:\mathcal{Q} \times \mathcal{Q} \to \mathcal{B}$ such that for all $r,s\in \mathcal{Q}$
    \begin{equation*}
d(r,s)=\begin{cases}
(0,0), &  \text{if } r=s\\
(1,0), &  \text{if } r>s\\
(0,1), &  \text{if } r<s.
\end{cases}
\end{equation*}
Then, $(\mathcal{Q},d)$ is a quasi-cone metric space.
\end{example}

Then, we define the set $\mathcal{S}$, referred to as sequentially compact, to characterize the existence of a best approximation:
\begin{definition}[See Definition 1.8 of \cite{Yaying2016}]
    \label{def5}
    Let $(\mathcal{Q},d)$ be a quasi-cone metric space. We say that a set $\mathcal{S}\subseteq Q$ is forward (resp. backward) sequentially compact if every sequence in $\mathcal{S}$ has a forward (resp. backward) convergent subsequence to an element of $\mathcal{S}$.
\end{definition}

Now we define the Chebyshev subset of $Q$, which we will use to determine the uniqueness of the best approximation:
\begin{definition}
    \label{def5}
    Let $(\mathcal{Q},d)$ be a quasi-cone metric space. We say that a set $H\subseteq Q$, $H\neq \emptyset$ is a forward (resp. backward) Chebyshev subset of $\mathcal{Q}$ if the forward (resp. backward) best approximation sets has only one element for all $q\in \mathcal{Q}$. 
\end{definition}

\begin{definition}
    Let $(\mathcal{Q},d)$ be a quasi-cone metric space. We say that a set $H\subseteq Q$, $H\neq \emptyset$ is forward (resp. backward) quasi Chebyshev subset of $\mathcal{Q}$ if the forward (resp. backward) best approximation sets is a forward sequentially compact (resp. backward) subset of $\mathcal{Q}$ for all $q\in \mathcal{Q}$.
\end{definition}

\begin{definition}
    Let $Q$ be a real vector space and $(\mathcal{Q},d)$ be a quasi-cone metric space. We say that a set $H\subseteq Q$, $H\neq \emptyset$ is forward (resp. backward) pseudo Chebyshev subset of $\mathcal{Q}$ if the forward (resp. backward) best approximation sets does not contains infinitely many linearly independent elements for all $q\in Q$.
\end{definition}
\section*{The structure of the paper}
The structure of the paper is organized as follows. In section \ref{sec2}, we present the definition of forward and backward best approximation in Definition \ref{def:fba} and \ref{def:bba}, then we demonstrate the concept in Example \ref{ex3} and \ref{ex4}. Next, the Theorem \ref{theorem1} and \ref{theorem2} is given as our main result that provide conditions for a set to be a subset of the best approximation set, with supporting lemmas in Lemma \ref{lemma:supp1} and \ref{lemma:supp2}. Additionally we provide some theorem concerning the Chebyshev set which characterize the uniqueness of the forward (backward) best approximations in Theorem \ref{theo3} and \ref{theo4}.  Finally, we present Theorems \ref{thm:fqC}, \ref{thm:bqC}, \ref{thm:fpqC}, and \ref{thm:bpqC} which characterize the forward (backward) quasi Chebyshev and pseudo quasi Chebyshev sets.

\section{Main Results}\label{sec2}
In this section, we present some of our results. \\
First, let us define the best approximation on quasi-cone metric spaces.
\begin{definition}[Forward Best Approximation]\label{def:fba}
Let $(\mathcal{Q},d)$ be a quasi-cone metric space, $H$ be a non-empty subset of $\mathcal{Q}$ and $q\in \mathcal{Q}$. If $h_{f} \in H$ and $d(q,h_f)\preceq d(q,h), \forall h \in H
$, then $h_f$ is an element of forward best approximation to $q$. We denote the set of all forward best approximations to $q$ in $H$ by $\mathcal{P}_{{H}_{f}}(q)$.
\end{definition}

\begin{definition}[Backward Best Approximation]\label{def:bba}
Let $(\mathcal{Q},d)$ be a quasi-cone metric space, $H$ be a non-empty subset of $\mathcal{Q}$ and $q\in \mathcal{Q}$. If $h_{b} \in H$ and $d(h_b,q)\preceq d(h,q), \forall h \in H$, then $h_b$ is an element of backward best approximation to $q$. We denote the set of all backward best approximations to $q$ in $H$ by $\mathcal{P}_{{H}_{b
}}(q)$.
\end{definition}

This following example demonstrates the concept.

\begin{example}
    \label{ex3}
    Suppose $\mathcal{Q}=\mathbb{R}, \mathcal{B} = \mathbb{R}^2$, and $\mathcal{P}=\{ (a,b) \in \mathcal{B}:a,b \ge 0\}$. Let us define $d:\mathcal{Q} \times \mathcal{Q} \to \mathcal{B}$ such that for all $r,s\in \mathcal{Q}$
    \begin{equation*}
d(r,s)=\begin{cases}
(0,0), &  \text{if } r= s\\
(1,0), & \text{if } r>s\\
(0,1), &  \text{if } r<s.
\end{cases}
\end{equation*}
Then, $(\mathcal{Q},d)$ is a quasi-cone metric space. We define a mapping $F:\mathcal{Q}\to \mathcal{Q}$ by $F(\beta)=\beta^2$ where $\beta \in \mathbb{R}$. Let $H:=(-\infty,0)$, then $h_f=-1$ is an element of  forward best approximations to $F$ in $H$ since $d(F,-1)=(1,0)=d(F,h)$ for all $h\in H$. Also, for all we have $\mathcal{P}_{H_f}=H$ is the set of all forward best approximations to $F$ in $H$, since for any $h_f\in \mathcal{P}_{H_f}$ one has $d(F,h_f)=(1,0)=d(F,h)$ for all $h\in H$.
\end{example}
\begin{example}
    \label{ex4}
    Suppose $\mathcal{Q}=\mathbb{R}, \mathcal{B} = \mathbb{R}^2$, $\mathcal{P}=\{ (a,b) \in \mathcal{B}:a,b \ge 0\}$, and $\alpha>0$. Let us define $d:\mathcal{Q} \times \mathcal{Q} \to \mathcal{B}$ such that for all $r,s\in \mathcal{Q}$
    \begin{equation*}
d(r,s)=\begin{cases}
(r-s,\alpha(r-s)), &  \text{if } r\geq s\\
(\alpha,1), &  \text{if } r<s.
\end{cases}
\end{equation*}
Then, $(\mathcal{Q},d)$ is a quasi-cone metric space. We define a mapping $F:\mathcal{Q}\to \mathcal{Q}$ by $F(\beta)=\beta$ where $\beta \in \mathbb{R}$. Let $H:=[0,2]$, then for $\beta>2$, $h_f=2$ is an element of forward best approximations to $F$ in $H$ since $$d(F,2)=(\beta-2,\alpha(\beta-2))\preceq (\beta-h,\alpha(\beta-h))=d(F,h)$$ for all $h\in H$. Furthermore, for $\beta<0$, $\mathcal{P}_{H_f}=H$ is the set of all forward best approximations to $F$ in $H$. For $\beta\in[0,2]$, then $\mathcal{P}_{H_f}=\{\beta\}$, and for $\beta>2$, then $\mathcal{P}_{H_f}=\{2\}$.
\end{example}

\section*{The Best Approximation Set}
In this part, we establish necessary and sufficient criteria for an element to qualify as a best approximation. We also provide conditions for a set to be a subset of the best approximation set in a quasi-cone metric space, both for forward and backward approximations. \\
Our main contributions regarding forward best approximation are summarized in the following theorem. This result provides a necessary and sufficient characterization for a set to be contained within the best approximation set.

\begin{theorem} \label{theorem1}
Let $(\mathcal{Q},d)$ be a quasi-cone metric space,  and let $H$ be a non-empty subset of $\mathcal{Q}$ and $q\in \mathcal{Q}$. Then $\mathcal{M}\subseteq \mathcal{P}_{H_f}(q)$ if and only if there exists a function $f:\mathcal{Q}\to \mathcal{B}$ such that $f(m_f)=d(q,m_f), f_{m_f}(H):=\{f(h)-f(m_f) ~:~ h\in H\}\subseteq \mathcal{P}$, and $f_{d_f}
(H):= \{d(q,h)-f(h):h\in H\}\subseteq \mathcal{P}$ for all $m_f \in \mathcal{M}$.
\end{theorem}

We first present a supporting lemma. The following lemma establishes necessary and sufficient conditions for an element of a set to belong to the best approximation sets.
\begin{lemma} \label {lemma:supp1}
Let $(\mathcal{Q},d)$ be a quasi-cone metric space, and let $H$ be a non-empty subset of $\mathcal{Q}$ and $q\in \mathcal{Q}$. Then $h_{f} \in H$ is a forward best approximation to $q\in \mathcal{Q}$ (i.e. $h_{f}\in \mathcal{P}_{H_{f}}(q)$) if and only if there exists a function $f:\mathcal{Q}\to \mathcal{B}$ such that $f(h_f)=d(q,h_f), f(h)\succeq f(h_f)$, and $d(q,h)\succeq f(h)$ for all $h \in H$.
\end{lemma}

\begin{proof}
Suppose that $h_f\in \mathcal{P}_{H_f}(q)$, we can define $f:\mathcal{Q}\to \mathcal{B}$ by $f(x)=d(q,x)$. Then for all $h\in H$ we have $d(q,h)=f(h)$ and  $$f(h)=d(q,h)\succeq d(q,h_f) =f(h_f),$$
from the definition of the forward best approximation.
  
Next, suppose that there exists a function $f:\mathcal{Q}\to \mathcal{B}$ such that $f(h_f)=d(q,h_f), f(h)\succeq f(h_f)$, and $d(q,h)\succeq f(h)\in \mathcal{P}$ for all $h \in H$. Then for all $h\in H$, one has
$$d(q,h_f)=f(h_f)\preceq f(h)\preceq d(q,h),$$
implying $h_f\in \mathcal{P}_{H_f}(q)$.
\end{proof}
We now present the proof of Theorem \ref{theorem1}.
\begin{proof}[Proof of Theorem \ref{theorem1}]
    Let us assume that there exists a function $f:\mathcal{Q}\to \mathcal{B}$ such that $f(m_f)=d(q,m_f), f_{m_f}(H)\subseteq \mathcal{P}, f_{d_f}(H)\subseteq \mathcal{P}$ for all $m_f\in \mathcal{M}$. By using Lemma \ref{lemma:supp1}, we have $m_f\in \mathcal{P}_{H_f}(q)$ for all $m_f\in \mathcal{M}$. Thus, $\mathcal{M}\subseteq \mathcal{P}_{H_f}(q)$.

    For the converse part, suppose that $\mathcal{M}\subseteq \mathcal{P}_{H_f}(q)$ and choose any $m_{f_1}\in \mathcal{M}$. Using Lemma \ref{lemma:supp1}, there exists a function $f:\mathcal{Q}\to \mathcal{B}$ such that $f(m_{f_1})=d(q,m_{f_1}), f_{m_{f_1}}(H)\subseteq \mathcal{P}$, and $f_{d_f}(H)\subseteq \mathcal{P}$. Then, for $m_f\in \mathcal{M}$, we have $f_{m_{f_1}}(m_f)\in \mathcal{P}$ and $f_{d_f}(m_f)\in \mathcal{P}$. In other words, we have $f(m_f)\succeq f(m_{f_1})=d(q,m_{f_1})$,  and $d(q,m_f)\succeq f(m_f)$. Since $m_f\in \mathcal{P}_{H_f}(q)$, then $d(q,m_f)\preceq d(q,m_{f_1})$. Hence $d(q,m_f)\preceq f(m_f)\preceq d(q,m_f)$, implying $f(m_f)=d(q,m_f)$. Furthermore, for every $h\in H$, one has
    \begin{align*}
        f_{m_f}(h)=f(h)-f(m_f) = f(h) - d(q,m_f) = f(h)-d(q,m_{f_1}) = f_{m_{f_1}}(h)\in \mathcal{P}
    \end{align*}
    Thus, $f$ is the function that we expected.
\end{proof}

Next, we also present the result for the backward best approximation, where the proof is similar to the forward case.
\begin{lemma} \label{lemma:supp2}
Let $(\mathcal{Q},d)$ be a quasi-cone metric space,  and let $H$ be a non-empty subset of $\mathcal{Q}$ and $q\in \mathcal{Q}$. Then $h_{b} \in H$ is a backward best approximation to $q\in \mathcal{Q}$ (i.e. $h_{b}\in \mathcal{P}_{H_{b}}(q)$) if and only if there exists a function $f:\mathcal{Q}\to \mathcal{B}$ such that $f(h_b)=d(h_b,q), f(h)\succeq f(h_b)$, and $d(h,q)\succeq f(h)$ for all $h \in H$.
\end{lemma}
\begin{proof}
Similar to Lemma \ref{lemma:supp1}
\end{proof}

\begin{theorem} \label{theorem2}
Let $(\mathcal{Q},d)$ be a quasi-cone metric space, and let $H$ be a non-empty subset of $\mathcal{Q}$ and $q\in \mathcal{Q}$. Then $\mathcal{M}\subseteq \mathcal{P}_{H_b}(q)$ if and only if there exists a function $f:\mathcal{Q}\to \mathcal{B}$ such that $f(m_b)=d(m_b,q), f_{m_b}(H)\subseteq \mathcal{P}$, $f_{d_b}(H):=\{d(h,q)-f(h):h\in H\}\subseteq \mathcal{P}$ for all $m_b \in \mathcal{M}$.
\end{theorem}
\begin{proof}
Similar to Theorem \ref{theorem1}
\end{proof}
\section*{The Chebyshev Set}
In this part, we discuss 3 types of Chebyshev sets. The first one is the Chebyshev set, which means that the best approximation is unique for each element. Next, the uniqueness condition of the best approximation is weakened in the case of quasi-Chebyshev sets. Finally, we consider pseudo-Chebyshev sets, which apply to a real vector space $Q$, where the uniqueness condition of the best approximation is also weakened.

\begin{theorem}\label{theo3}
    Let $(\mathcal{Q},d)$ be a quasi-cone metric space, and let $H$ be a non-empty subset of $\mathcal{Q}$. Then $H$ is a forward Chebyshev subset of $\mathcal{Q}$ if and only if there don't exist $q\in \mathcal{Q}$, distinct elements $h_1,h_2\in H$, and function $f:\mathcal{Q}\to \mathcal{B}$ such that $f(h_i)=d(q,h_i)$, $f_{h_i}(H)\subseteq \mathcal{P}$, and $f_{d_f}(H)\subseteq \mathcal{P}$.
\end{theorem}
\begin{proof}
    Suppose that $H$ is not forward Chebyshev subset of $\mathcal{Q}$, then there exist $q\in \mathcal{Q}$ and distinct elements $h_1,h_2\in \mathcal{P}_{H_f}(q)\subseteq H$. Then, by Theorem \ref{theorem1}, there exists a function $f:\mathcal{Q}\to \mathcal{B}$ such that $f(h_i)=d(q,h_i)$, $f_{h_i}(H)\subseteq \mathcal{P}$ and $f_{d_f}(H)\subseteq \mathcal{P}$ for $i=1,2$. \\
    For the converse part, suppose that there exist $q\in \mathcal{Q}$, distinct elements $h_1,h_2\in H$, and a function $f:\mathcal{Q}\to \mathcal{B}$ such that $f(h_i)=d(q,h_i),~ f_{h_i}(H)\subseteq \mathcal{P}$, and $f_{d_f}(H)\subseteq \mathcal{P}$ for $i=1,2$. Then, by Theorem \ref{theorem1}, we have $h_1,h_2\in \mathcal{P}_{H_f}(q)$ which means $H$ is not a forward Chebyshev subset of $\mathcal{P}$.
\end{proof}
\begin{theorem}\label{theo4}
    Let $(\mathcal{Q},d)$ be a quasi-cone metric space, and let $H$ be a non-empty subset of $\mathcal{Q}$. Then $H$ is a backward Chebyshev subset of $\mathcal{Q}$ if and only if there don't exist $q\in \mathcal{Q}$, distinct elements $h_1,h_2\in H$, and function $f:\mathcal{Q}\to \mathcal{B}$ such that $f(h_i)=d(h_i,q)$, $f_{h_i}(H)\subseteq \mathcal{P}$, and $f_{d_b}(H)\subseteq \mathcal{P}$.
\end{theorem}
\begin{proof}
Similar to Theorem \ref{theo3}.
\end{proof}
\begin{theorem}\label{thm:fqC}
    Let $(\mathcal{Q},d)$ be a quasi-cone metric space, and let $H$ be a non-empty subset of $\mathcal{Q}$. Then $H$ is a forward quasi Chebyshev subset of $\mathcal{Q}$ if and only if there don't exist $q\in \mathcal{Q}$, a sequence $(h_n)\subseteq H$ without a $f$-convergent subsequence and a function $f:\mathcal{Q}\to \mathcal{B}$ such that $f(h_n)=d(q,h_n), f_{h_n}(H)\subseteq \mathcal{P}$, and $f_{d_f}(H)\subseteq \mathcal{P}$.
\end{theorem}
\begin{proof}
    Suppose that $H$ is not forward quasi Chebyshev subset of $\mathcal{Q}$. Then $\mathcal{P}_{H_f}(q)$ is not forward sequentially compact subset of $\mathcal{Q}$. It implies that there exist $q\in \mathcal{Q}$ and a sequence $(h_n)\in \mathcal{P}_{H_f}(q)$ without a $f$-convergent subsequence. Thus, by Theorem \ref{theorem1}, there exist $f:\mathcal{Q}\to \mathcal{B}$ such that $f(h_n)=d(q,h_n),~ f_{h_n}(H)\subseteq \mathcal{P}$, and $f_{d_f}(H)\subseteq \mathcal{P}$ for all $n\in\natural$.\\

    Next, suppose that there exist $q\in \mathcal{Q}$, a sequence $(h_n)\subseteq H$ without a $f$-convergent subsequence and a function $f:\mathcal{Q}\to \mathcal{B}$ such that $f(h_n)=d(q,h_n), f_{h_n}(H)\subseteq \mathcal{P}$, and $f_{d_f}(H)\subseteq \mathcal{P}$ for all $n\in \natural$. By Theorem \ref{theorem1}, $h_n\in \mathcal{P}_{H_f}(q)$ for all $n\in \natural$ and it implies that $\mathcal{P}_{H_f}(q)$ is not forward sequentially compact subset of $\mathcal{Q}$. Therefore, $H$ is not forward quasi Chebyshev subset of $\mathcal{Q}$.
\end{proof}
\begin{theorem}\label{thm:bqC}
    Let $(\mathcal{Q},d)$ be a quasi-cone metric space, and let $H$ be a non-empty subset of $\mathcal{Q}$. Then $H$ is a backward quasi Chebyshev subset of $\mathcal{Q}$ if and only if there don't exist $q\in \mathcal{Q}$, a sequence $(h_n)\subseteq H$ without a $b$-convergent subsequence and a function $f:\mathcal{Q}\to \mathcal{B}$ such that $f(h_n)=d(h_n,q), f_{h_n}(H)\subseteq \mathcal{P}$, and $f_{d_b}(H)\subseteq \mathcal{P}$.
\end{theorem}
\begin{proof}
Similar to Theorem \ref{thm:fqC}.
\end{proof}
\begin{theorem}\label{thm:fpqC}
    Let $\mathcal{Q}$ be a real vector space, and let $(\mathcal{Q},d)$ be a quasi-cone metric space, and let $H$ be a non-empty subset of $\mathcal{Q}$. Then $H$ is a forward pseudo Chebyshev subset of $\mathcal{Q}$ if and only if there don't exist $q\in \mathcal{Q}$, infinitely many linearly independent elements $\{h_n\}\subseteq H$, and a function $f:\mathcal{Q}\to \mathcal{B}$ such that $f(h_n)=d(q,h_n), f_{h_n}(H)\subseteq \mathcal{P}$, and $f_{d_f}(H)\subseteq \mathcal{P}$.
\end{theorem}
\begin{proof}
    Suppose that $H$ is not backward pseudo Chebyshev subset of $\mathcal{Q}$. Then, there exist $q\in \mathcal{Q}$ and infinitely many linearly independent elements $\{h_n\}\subseteq \mathcal{P}_{H_f}(q)$. Thus, by Theorem \ref{theorem1}, there exist $f:\mathcal{Q}\to \mathcal{B}$ such that $f(h_n)=d(q,h_n),~ f_{h_n}(H)\subseteq \mathcal{P}$, and $f_{d_f}(H)\subseteq \mathcal{P}$ for all $n\in\natural$.\\

    Next, suppose that there exist $q\in \mathcal{Q}$, infinitely many linearly independent elements $\{h_n\}\subseteq H$ and a function $f:\mathcal{Q}\to \mathcal{B}$ such that $f(h_n)=d(q,h_n), f_{h_n}(H)\subseteq \mathcal{P}$, and $f_{d_f}(H)\subseteq \mathcal{P}$ for all $n\in \natural$. By Theorem \ref{theorem1}, $h_n\in \mathcal{P}_{H_f}(q)$ for all $n\in \natural$ and it implies that $H$ is not backward pseudo Chebyshev subset of $\mathcal{Q}$.
\end{proof}

\begin{theorem}\label{thm:bpqC}
    Let $\mathcal{Q}$ be a real vector space, and let $(\mathcal{Q},d)$ be a quasi-cone metric space, and let $H$ be a non-empty subset of $\mathcal{Q}$. Then $H$ is a backward pseudo Chebyshev subset of $\mathcal{Q}$ if and only if there don't exist $q\in \mathcal{Q}$, infinitely many linearly independent elements $(h_n)\subseteq H$, and a function $f:\mathcal{Q}\to \mathcal{B}$ such that $f(h_n)=d(h_n,q), f_{h_n}(H)\subseteq \mathcal{P}$, and $f_{d_b}(H)\subseteq \mathcal{P}$.
\end{theorem}
\begin{proof}
Similar to Theorem \ref{thm:fpqC}
\end{proof}

 \bibliographystyle{elsarticle-num} 
 \bibliography{reference}

\end{document}